\documentclass{article}

\usepackage{amssymb,amsmath,amsthm,amsfonts,enumerate}

\newtheorem{theorem}{Theorem}
\newtheorem{corollary}[theorem]{Corollary}
\newtheorem{definition}[theorem]{Definition}

\newtheorem{problem}{Problem}
\newtheorem{remark}[theorem]{Remark}

\begin{document}

\title{Fractional Variational Calculus with Classical
and Combined Caputo Derivatives\thanks{Part
of the first author's Ph.D., which is carried out
at the University of Aveiro under the Doctoral Programme
\emph{Mathematics and Applications}
of Universities of Aveiro and Minho. Submitted 30-Nov-2010;
accepted 14-Jan-2011; for publication 
in \emph{Nonlinear Analysis Series A: Theory, Methods \& Applications}.}}

\author{Tatiana Odzijewicz$^{1}$\\
        \texttt{tatianao@ua.pt}
\and
        Agnieszka B. Malinowska$^{2}$\\
        \texttt{abmalinowska@ua.pt}
\and
        Delfim F. M. Torres$^{1}$\\
        \texttt{delfim@ua.pt}}

\date{$^1$Department of Mathematics\\
University of Aveiro\\
3810-193 Aveiro, Portugal\\[0.3cm]
$^2$Faculty of Computer Science\\
Bia{\l}ystok University of Technology\\
15-351 Bia\l ystok, Poland}

\maketitle


\begin{abstract}
We give a proper fractional extension of the classical calculus
of variations by considering variational functionals with a
Lagrangian depending on a combined Caputo fractional derivative
and the classical derivative. Euler--Lagrange equations to
the basic and isoperimetric problems are proved,
as well as transversality conditions.

\bigskip


\bigskip


\end{abstract}


\section{Introduction}

Fractional calculus (FC) is a generalization of (integer) differential calculus,
in the sense it deals with derivatives of real or complex order.
FC was born on 30th September 1695. On that day, L'H\^{o}pital
wrote a letter to Leibniz, where he asked about Leibniz's notation
of $n$th order derivatives of a linear function.
L'H\^{o}pital wanted to know the result for the derivative of order $n=1/2$.
Leibniz replied that ``\emph{one day, useful consequences will be drawn}''
and, in fact, his vision became a reality.  The study of non-integer order derivatives
rapidly became a very attractive subject to mathematicians,
and many different forms of fractional (\textrm{i.e.}, non-integer)
derivative operators were introduced:
the Grunwald--Letnikow, Riemann--Liouville, Caputo \cite{Hilfer,Kilbas,Podlubny},
and the more recent notions of Cresson \cite{Cresson}, Jumarie \cite{Jumarie},
or Klimek \cite{Klimek}.

The calculus of variations with fractional
derivatives was born in 1996 with the work of Riewe,
to better describe nonconservative systems in mechanics
\cite{CD:Riewe:1996,CD:Riewe:1997}. It is a subject
of strong current research due to its many applications
in science and engineering, including mechanics,
chemistry, biology, economics, and control theory
(see, \textrm{e.g.}, the recent papers
\cite{FrMult,MyID:154,MyID:147,MyID:152,MyID:179,FrTor,%
withBasiaRachid,NatFr,comDorota,MyID:181}).\footnote{The
literature on \emph{fractional variational calculus}
is vast, and we do not try to provide here a comprehensive
review on the subject. We give only some representative references
from 2010 and 2011. Other references can be found therein.}

Following \cite{Tatiana:Spain2010},
we consider here that the highest derivative
in the Lagrangian is of integer order.
The main advantage of our formulation,
with respect to the ``pure'' fractional
approach adopted in the literature,
is that the classical results of variational calculus
can now be obtained as particular cases. We recall that
the only possibility of obtaining the classical derivative $y'$ from
a fractional derivative $D^\alpha y$, $\alpha \in (0,1)$,
is to take the limit when $\alpha$ tends to one.
However, in general such a limit does not exist \cite{Ross:Samko:Love}.
Differently from \cite{Tatiana:Spain2010},
where the fractional problems are considered in the sense
of Riemann--Liouville, we consider here
combined Caputo derivatives $^{C}D^{\alpha,\beta}_{\gamma}$.
The operator $^{C}D^{\alpha,\beta}_{\gamma}$
extends the Caputo fractional derivatives,
and was introduced for the first time in \cite{ComCa}
as a useful tool in the description of some nonconservative models
and more general classes of variational problems.
More precisely, we investigate here problems of the calculus
of variations with integrands depending
on the independent variable $x$, an unknown vector-function $\textbf{y}$,
its integer order derivative $\textbf{y}'$,
and a fractional derivative $^{C}D^{\alpha,\beta}_{\gamma}\textbf{y}$
given as a convex combination
of the left Caputo fractional derivative of order $\alpha$
and the right Caputo fractional derivative of order $\beta$.

The paper is organized as follows. Section~\ref{sec:PFC}
presents the basic definitions and facts needed in the sequel.
Our results are then stated and proved in Section~\ref{sec-FrCVDI}.
We discuss the fundamental concepts of a variational calculus such as the
Euler--Lagrange equations for the basic (Theorem~\ref{theorem:EL})
and isoperimetric (Theorem~\ref{theorem:isop}) problems, as well as
transversality conditions (Theorem~\ref{theorem:nbc}).
We end with an illustrative example of the results
of the paper (Section~\ref{sec:ex}).


\section{Preliminaries on Fractional Calculus}
\label{sec:PFC}

In this section we present some basic necessary facts on fractional calculus.
For more on the subject and applications,
we refer the reader to the books \cite{Hilfer,Kilbas,Podlubny}.

\begin{definition}[Riemann--Liouville fractional integrals]
Let $f\in L_1\left([a,b]\right)$ and $0<\alpha<1$.
The left Riemann--Liouville Fractional Integral (RLFI)
of order $\alpha$ of a function $f$ is defined by
\begin{equation*}
_{a}\textsl{I}_x^\alpha f(x)
:=\frac{1}{\Gamma(\alpha)}\int_a^x(x-t)^{\alpha-1}f(t)dt,
\end{equation*}
and the right RLFI by
\begin{equation*}
_{x}\textsl{I}_b^\alpha f(x)
:=\frac{1}{\Gamma(\alpha)}\int_x^b(t-x)^{\alpha-1}f(t)dt,
\end{equation*}
for all $x\in[a,b]$.
\end{definition}

\begin{definition}[Left and right Riemann--Liouville fractional derivatives]
The left Riemann--Liouville Fractional Derivative (RLFD)
of order $\alpha$ of a function $f$, denoted by $_{a}\textsl{D}_x^\alpha f$,
is defined by
\begin{equation*}
_{a}\textsl{D}_x^\alpha f(x)
:= \frac{d}{dx}{_{a}}\textsl{I}_x^{1-\alpha}f(x)
=\frac{1}{\Gamma(1-\alpha)}\frac{d}{dx}\int_a^x(x-t)^{-\alpha}f(t)dt,
\end{equation*}
$x\in [a,b]$. Similarly, the right RLFD of order $\alpha$
of a function $f$, denoted by $_{x}\textsl{D}_b^\alpha f$,
is defined by
\begin{equation*}
_{x}\textsl{D}_b^\alpha f(x)
:= -\frac{d}{dx}{_{x}}\textsl{I}_b^{1-\alpha}f(x)
= \frac{-1}{\Gamma(1-\alpha)}\frac{d}{dx}\int_x^b(t-x)^{-\alpha}f(t)dt,
\end{equation*}
$x\in [a,b]$.
\end{definition}

\begin{definition}[Caputo fractional derivatives]
Let $f\in AC\left([a,b]\right)$, where $AC\left([a,b]\right)$
represents the space of absolutely continuous functions
on the interval $[a,b]$. The left Caputo Fractional Derivative (CFD) is defined by
\begin{equation*}
^{C}_{a}\textsl{D}_x^\alpha f(t)
:= {_{a}}\textsl{I}_x^{1-\alpha}\left(\frac{d}{dt}f\right)(x)
= \frac{1}{\Gamma(1-\alpha)}\int_a^x(x-t)^{-\alpha}\frac{d}{dt}f(t)dt,
\end{equation*}
$x\in [a,b]$, and the right CFD by
\begin{equation*}
_{x}\textsl{D}_b^\alpha f(x)
:= {_{x}}\textsl{I}_b^{1-\alpha}\left(-\frac{d}{dt}f\right)(x)
= \frac{-1}{\Gamma(1-\alpha)}\int_x^b(t-x)^{-\alpha}\frac{d}{dt}f(t)dt,
\end{equation*}
$x\in [a,b]$, where $\alpha$ is the order of the derivative.
\end{definition}

\begin{theorem}[Fractional integration by parts \cite{Kilbas}]
Let $p\geq1$, $q\geq1$, and $\frac{1}{p}+\frac{1}{q}\leq1+\alpha$.
If $g\in L_p\left([a,b]\right)$ and $f\in L_q\left([a,b]\right)$,
then the following formula for integration by parts holds:
$$
\int\limits_a^bg(x){_{a}}\textsl{I}_x^{\alpha}f(x)dx
=\int\limits_a^bf(x){_{x}}\textsl{I}_b^{\alpha}g(x)dx.
$$
\end{theorem}

\begin{definition}[The combined fractional derivative
$^{C}D^{\alpha,\beta}_{\gamma}$ \cite{ComCa}]
Let $\alpha,\beta\in(0,1)$ and $\gamma\in[0,1]$.
The combined fractional derivative operator $^{C}D^{\alpha,\beta}_{\gamma}$
is given by
$$
^{C}D^{\alpha,\beta}_{\gamma}
:=\gamma ^{C}_{a}D^{\alpha}_{x}+(1-\gamma) ^{C}_{x}D^{\beta}_{b}.
$$
\end{definition}

\begin{remark}
\label{remarl:pc:comD}
The combined fractional derivative coincides
with the right CFD in the case $\gamma = 0$, \textrm{i.e.},
$^{C}D^{\alpha,\beta}_{0}f(x)= ^{C}_{x}D^{\alpha}_{b}f(x)$.
For $\gamma = 1$ one gets the left CFD:
$^{C}D^{\alpha,\beta}_{1}f(x)= ^{C}_{a}D^{\alpha}_{x}f(x)$.
\end{remark}

For $\textbf{f}=[f_1,\ldots,f_N]:[a,b]\rightarrow\mathbb{R}^N$,
$N\in\mathbb{N}$, and $f_i\in AC\left([a,b]\right)$,
$i=1,\ldots,N$, we put
\begin{equation*}
^{C}D^{\alpha,\beta}_{\gamma}\textbf{f}(x)
:=\left[^{C}D^{\alpha,\beta}_{\gamma}f_1(x),
\ldots, ^{C}D^{\alpha,\beta}_{\gamma}f_N(x)\right].
\end{equation*}

In the discussion to follow, we also need the following formula
for fractional integrations by parts \cite{ComCa}:
\begin{multline}
\label{eq:fracIP:C}
\int\limits_a^b g(x) ^{C}D^{\alpha,\beta}_{\gamma}f(x)dx
= \int\limits_a^b f(x)D^{\beta,\alpha}_{1-\gamma}g(x)dx\\
+\left[ \gamma f(x)_{x}\textsl{I}_b^{1-\alpha} g(x)
- \left(1-\gamma\right) f(x)_{a}\textsl{I}_x^{1-\beta} g(x)\right]_{x=a}^{x=b},
\end{multline}
where $D^{\beta,\alpha}_{1-\gamma}:=(1-\gamma)_{a}\textsl{D}_x^\beta
+\gamma _{x}\textsl{D}_b^\alpha$.


\section{Main Results}
\label{sec-FrCVDI}

Consider the following functional:
\begin{equation}
\label{eq:FU}
\mathcal{J}\textsl{}\left(\textbf{y}\right)
=\int\limits_a^b
L\left(x,\textbf{y}(x),\textbf{y}'(x),{^{C}}D^{\alpha,\beta}_{\gamma}\textbf{y}(x)\right)dx,
\end{equation}
where $x\in[a,b]$ is the independent variable;
$\textbf{y}(x)\in\mathbb{R}^N$ is a real vector variable;
$\textbf{y}'(x)\in\mathbb{R}^N$ with $\textbf{y}'$ the
first derivative of $\textbf{y}$; ${^{C}}D^{\alpha,\beta}_{\gamma}\textbf{y}(x)
\in\mathbb{R}^N$ stands for the combined fractional derivative of $\textbf{y}$
evaluated in $x$; and $L\in C^1\left([a,b]\times\mathbb{R}^{3N};\mathbb{R}\right)$.
Let $\textbf{D}$ denote the set of all functions
$\textbf{y}:[a,b]\rightarrow \mathbb{R}^N$ such that $\textbf{y}'$
and $^{C}D^{\alpha,\beta}_{\gamma}\textbf{y}$ exist and are continuous
on the interval $[a,b]$. We endow $\textbf{D}$ with the norm
\begin{equation*}
\left\|\textbf{y}\right\|_{1,\infty}
:= \max\limits_{a\leq x\leq b}\left\|\textbf{y}(x)\right\|
+\max\limits_{a\leq x\leq b}\left\|\textbf{y}'(x)\right\|
+\max\limits_{a\leq x\leq b}\left\| ^{C}D^{\alpha,\beta}_{\gamma}\textbf{y}(x)\right\|,
\end{equation*}
where $\left\|\cdot\right\|$ is a norm in $\mathbb{R}^N$.
Along the work, we denote by $\partial_{i}K$, $i=1,\ldots,M$ $(M\in\mathbb{N})$, the partial
derivative of a function $K:\mathbb{R}^M\rightarrow\mathbb{R}$ with respect to its $i$th argument.
Let $\lambda\in\mathbb{R}^r$. For simplicity of notation
we introduce the operators $\left[\cdot\right]^{\alpha,\beta}_{\gamma}$
and $_{\lambda}\left\{\cdot\right\}^{\alpha,\beta}_{\gamma}$ defined by
\begin{gather*}
\left[\textbf{y}\right]^{\alpha,\beta}_{\gamma}(x)
:=\left(x,\textbf{y}(x),\textbf{y}'(x),
^{C}D^{\alpha,\beta}_{\gamma}\textbf{y}(x)\right),\\
_{\lambda}\left\{\textbf{y}\right\}^{\alpha,\beta}_{\gamma}(x)
:=\left(x,\textbf{y}(x),\textbf{y}'(x),
^{C}D^{\alpha,\beta}_{\gamma}\textbf{y}(x),
\lambda_1,\ldots,\lambda_r\right).
\end{gather*}


\subsection{The Euler--Lagrange equation}
\label{sub:sec:FrIP}

We begin with the following problem of the fractional calculus of variations.
\begin{problem}
\label{problem:1} Find a function $y\in\textnormal{\textbf{D}}$ for
which the functional \eqref{eq:FU}, \textrm{i.e.},
\begin{equation}
\label{eq:Funct} \mathcal{J}\left(y\right)
=\int\limits_a^bL\left[y\right]^{\alpha,\beta}_{\gamma}(x)dx,
\end{equation}
subject to given boundary conditions
\begin{equation}
\label{eq:BoundCond} y(a)=\textbf{y}^a, \quad y(b) =\textbf{y}^b,
\end{equation}
$\textbf{y}^a,\textbf{y}^b\in\mathbb{R}^N$, achieves a minimum.
\end{problem}

\begin{definition}[Admissible function]
\label{adm:f} A function $y\in\textnormal{\textbf{D}}$ that
satisfies all the constraints of a problem is said to be admissible
to that problem. The set of admissible functions is denoted by
$\mathcal{D}$.
\end{definition}

\begin{remark}
For Problem~\ref{problem:1} the constraints
mentioned in Definition~\ref{adm:f} are the
boundary conditions \eqref{eq:BoundCond}.
\end{remark}

We now define what is meant by minimum of
$\mathcal{J}$ on $\mathcal{D}$.

\begin{definition}[Local minimizer]
A function $\overline{y}\in\mathcal{D}$ is said to be a local
minimizer to $\mathcal{J}$ on $\mathcal{D}$ if there exists some
$\delta>0$ such that
\begin{equation*}
\mathcal{J}\left(\overline{y}\right) -\mathcal{J}\left(y\right) \leq
0
\end{equation*}
for all functions $y\in\mathcal{D}$ with $\left\|y
-\overline{y}\right\|_{1,\infty}<\delta$.
\end{definition}

Similarly to the classical calculus of variations,
a necessary optimality condition to Problem~\ref{problem:1}
is based on the concept of variation.

\begin{definition}[First variation]
The first variation of $\mathcal{J}$ at
$y\in\textnormal{\textbf{D}}$ in the direction
$\textnormal{\textbf{h}}\in\textnormal{\textbf{D}}$ is defined by
\begin{equation*}
\delta \mathcal{J}\left(y;\textnormal{\textbf{h}}\right)
:=\lim\limits_{\epsilon\rightarrow 0}\frac{\mathcal{J}\left(y
+\epsilon
\textnormal{\textbf{h}}\right)-\mathcal{J}\left(y\right)}{\epsilon}
=\left.\frac{\partial}{\partial\epsilon}\textit{J}\left(y
+\epsilon\textnormal{\textbf{h}}\right)\right|_{\epsilon=0},
\end{equation*}
provided the limit exists.
\end{definition}

\begin{definition}[Admissible variation]
An admissible variation at $y\in\mathcal{D}$ for $\mathcal{J}$ is a
direction $\textnormal{\textbf{h}}\in\textnormal{\textbf{D}}$,
$\textnormal{\textbf{h}}\neq 0$, such that
\begin{itemize}
\item $\delta \mathcal{J}\left(y;
\textnormal{\textbf{h}}\right)$ exists; and
\item $y+\epsilon\textnormal{\textbf{h}}
\in\mathcal{D}$ for all sufficiently small $\epsilon$.
\end{itemize}
\end{definition}

\begin{theorem}[see, \textrm{e.g.}, \cite{Troutman}]
\label{thm:Troutman} Let $\mathcal{J}$ be a functional defined on
$\mathcal{D}$. Suppose that $y$ is a local minimizer to
$\mathcal{J}$ on $\mathcal{D}$. Then $\delta \mathcal{J}\left(y;
\textnormal{\textbf{h}}\right)=0$ for each admissible variation
$\textnormal{\textbf{h}}$ at $y$.
\end{theorem}

We now state and prove the Euler--Lagrange
equations for Problem~\ref{problem:1}.

\begin{theorem}
\label{theorem:EL} If $y=\left(y_1,\ldots,y_N\right)$ is a local
minimizer to Problem~\ref{problem:1}, then $y$ satisfies the system
of $N$ Euler--Lagrange equations
\begin{equation}
\label{eq:EL}
\partial_i L \left[y\right]^{\alpha,\beta}_{\gamma}(x)
-\frac{d}{dx}\partial_{N+i}L
\left[y\right]^{\alpha,\beta}_{\gamma}(x)
+D^{\beta,\alpha}_{1-\gamma}\partial_{2N+i}L
\left[y\right]^{\alpha,\beta}_{\gamma}(x)=0,
\end{equation}
$i=2,\ldots,N+1$, for all $x\in[a,b]$.
\end{theorem}

\begin{proof}
Suppose that $y$ is a solution to Problem~\ref{problem:1} and let
$\textbf{h}$ be an arbitrary admissible variation for this problem,
\textrm{i.e.},
\begin{equation*}
h_i(a)=h_i(b)=0, \quad i=1,\ldots,N.
\end{equation*}
According with Theorem~\ref{thm:Troutman}, a necessary condition
for $\textbf{y}$ to be a minimizer is given by
\begin{equation*}
\frac{\partial}{\partial\epsilon}\mathcal{J}\left.\left(y
+\epsilon\textbf{h}\right)\right|_{\epsilon=0}=0,
\end{equation*}
that is,
\begin{equation}
\label{eq:1}
\begin{split}
& \int\limits_a^b\left[\sum\limits_{i=2}^{N+1}\partial_i L
\left[y\right]^{\alpha,\beta}_{\gamma}(x)h_{i-1}(x) +
\sum\limits_{i=2}^{N+1}\partial_{N+i}
L\left[y\right]^{\alpha,\beta}_{\gamma}(x)\frac{d}{dx}h_{i-1}(x)\right.\\
&\quad \left.+\sum\limits_{i=2}^{N+1}\partial_{2N+i}
L\left[y\right]^{\alpha,\beta}_{\gamma}(x)\left(^{C}D^{\alpha,
\beta}_{\gamma}h_{i-1}(x)\right)\right]dx=0.
\end{split}
\end{equation}
Using the integration by parts formulas, for the classical and
$^{C}D^{\alpha,\beta}_{\gamma}$ derivatives,
in the second and third term of the integrand, we obtain
\begin{equation}
\label{eq:2}
\begin{split}
&\int\limits_a^b\sum\limits_{i=2}^{N+1}\left[\partial_i
L\left[y\right]^{\alpha,\beta}_{\gamma}(x)
-\frac{d}{dx}\partial_{N+i} L
\left[y\right]^{\alpha,\beta}_{\gamma}(x) +
^{C}D^{\beta,\alpha}_{1-\gamma}\partial_{2N+i}
L\left[y\right]^{\alpha,\beta}_{\gamma}(x)\right]h_{i-1}(x)dx\\
&\quad+\left.\left[\sum\limits_{i=2}^{N+1}h_{i-1}(x)\partial_{N+i}
L\left[y\right]^{\alpha,\beta}_{\gamma}(x)\right]\right._{x=a}^{x=b}
+\gamma\left.\left[\sum\limits_{i=2}^{N+1}h_{i-1}(x)_{x}I^{1-\alpha}_{b}\partial_{2N+i}
L\left[y\right]^{\alpha,\beta}_{\gamma}(x)\right]\right._{x=a}^{x=b}\\
&
\quad-(1-\gamma)\left.\left[\sum\limits_{i=2}^{N+1}h_{i-1}(x)_{a}I^{1-\beta}_{x}\partial_{2N+i}
L\left[y\right]^{\alpha,\beta}_{\gamma}(x)\right]\right._{x=a}^{x=b}=0.
\end{split}
\end{equation}
Since $h_i(a)=h_i(b)=0$, $i=1,\ldots,N$, we get
\begin{equation*}
\int\limits_a^b\sum\limits_{i=2}^{N+1}\left[\partial_i L
\left[y\right]^{\alpha,\beta}_{\gamma}(x)
-\frac{d}{dx}\partial_{N+i}
L\left[y\right]^{\alpha,\beta}_{\gamma}(x) +
^CD^{\beta,\alpha}_{1-\gamma}\partial_{2N+i}
L\left[y\right]^{\alpha,\beta}_{\gamma}(x)\right]h_{i-1}(x)dx=0.
\end{equation*}
Equalities \eqref{eq:EL} follow from the application of the
fundamental lemma of the calculus of variations
(see, \textrm{e.g.}, \cite{Troutman}).
\end{proof}

When the Lagrangian $L$ does not depend on fractional derivatives,
then Theorem~\ref{theorem:EL} reduces to the classical result
(see, \textrm{e.g.}, \cite{Troutman}). The fractional Euler--Lagrange
equations via Caputo derivatives that one can find in the literature,
are also obtained as corollaries of Theorem~\ref{theorem:EL}.
The next result is obtained choosing a Lagrangian
that does not depend on the classical derivatives.

\begin{corollary}[Theorem~6 of \cite{ComCa}]
\label{Theo E-L1}
Let $\mathbf{y}=(y_1,\ldots,y_N)$ be a local
minimizer to problem
\begin{gather*}
\mathcal{J}(\mathbf{y})
=\int_a^b L\left(x,\mathbf{y}(x),{^CD^{\alpha,\beta}_{\gamma}} \mathbf{y}(x)\right)
\, dx \longrightarrow \min\\
\mathbf{y}(a)=\mathbf{y}^{a},
\quad \mathbf{y}(b)=\mathbf{y}^{b},
\end{gather*}
$\mathbf{y}^{a}$, $\mathbf{y}^{b}\in \mathbb{R}^N$.
Then, $\mathbf{y}$ satisfies the system of
$N$ fractional Euler--Lagrange equations
\begin{equation}
\label{E-L1}
\partial_i L[\mathbf{y}](x)+D^{\beta,\alpha}_{1-\gamma}
\partial_{N+i}L[\mathbf{y}](x)=0,
\end{equation}
$i=2,\ldots N+1$, for all $x\in[a,b]$.
\end{corollary}

If one considers $\gamma = 1$ (\textrm{cf.} Remark~\ref{remarl:pc:comD})
and $N = 1$ in Corollary~\ref{Theo E-L1}, then \eqref{E-L1} reduces to
the well known Caputo fractional Euler--Lagrange equation:
if $y$ is a local minimizer to problem
\begin{gather*}
\mathcal{J}(y)=\int_a^b
L\left(x,y(x),{^C_a D_x^{\alpha}}y(x)\right) \, dx \longrightarrow \min\\
y(a) = y_a, \quad y(b) = y_b,
\end{gather*}
then $y$ satisfies the fractional Euler--Lagrange equation
\begin{equation}
\label{eq:fEL:cC}
\partial_2 L\left(x,y(x),{^C_a D_x^{\alpha}}y(x)\right)
+{_x D_b^{\alpha}}\partial_3 L\left(x,y(x),{^C_a D_x^{\alpha}} y(x)\right)  = 0
\end{equation}
for all $x\in[a,b]$ (see, \textrm{e.g.}, \cite{FrTor:Caputo}).


\subsection{Transversality conditions}

Let $l\in\left\{1,\ldots,N\right\}$.
Assume now that in Problem~\ref{problem:1}
the boundary conditions \eqref{eq:BoundCond}
are substituted by
\begin{equation}
\label{eq:BoundCond2}
\textbf{y}(a)=\textbf{y}^a, \quad y_i(b)=y^b_i,
\ i=1,\ldots,N \textnormal{ for } i\neq l,
\textnormal{ and } y_l(b) \textnormal{ is free}
\end{equation}
or
\begin{equation}
\label{eq:BoundCond3}
\textbf{y}(a)=\textbf{y}^a, \quad y_i(b)=y^b_i,
\ i=1,\ldots,N \textnormal{ for }
i\neq l, \textnormal{ and } y_l(b)\leq y_l^b.
\end{equation}

\begin{theorem}
\label{theorem:nbc} If $y=\left(y_1,\ldots,y_N\right)$ is a solution
to Problem~\ref{problem:1} with either \eqref{eq:BoundCond2} or
\eqref{eq:BoundCond3} as boundary conditions instead of
\eqref{eq:BoundCond}, then $y$ satisfies the system of
Euler--Lagrange equations \eqref{eq:EL}. Moreover, under the
boundary conditions \eqref{eq:BoundCond2} the extra transversality
condition
\begin{multline}
\label{eq:tc:1stc}
\left[\partial_{N+l+1}L\left[y\right]^{\alpha,\beta}_{\gamma}(x)
+\gamma_{x}I^{1-\alpha}_{b}\partial_{2N+l+1}
L\left[y\right]^{\alpha,\beta}_{\gamma}(x)\right.\\
\left.-(1-\gamma)_{a}I^{1-\beta}_{x}\partial_{2N+l+1}
L\left[y\right]^{\alpha,\beta}_{\gamma}(x)\right]_{x=b}=0
\end{multline}
holds; under the boundary conditions \eqref{eq:BoundCond3}
the extra transversality condition
\begin{multline}
\label{eq:tc:2ndc}
\left[\partial_{N+l+1}L\left[y\right]^{\alpha,\beta}_{\gamma}(x)
+\gamma_{x}I^{1-\alpha}_{b}\partial_{2N+l+1}
L\left[y\right]^{\alpha,\beta}_{\gamma}(x)\right.\\
\left.-(1-\gamma)_{a}I^{1-\beta}_{x}\partial_{2N+l+1}
L\left[y\right]^{\alpha,\beta}_{\gamma}(x)\right]_{x=b} \leq 0
\end{multline}
holds, with \eqref{eq:tc:1stc} taking place if $y_l(b)<y_l^b$.
\end{theorem}

\begin{proof}
The fact that the system of Euler--Lagrange equations \eqref{eq:EL}
is satisfied is a simple consequence of the proof of Theorem~\ref{theorem:EL}
(one can always restrict to the subclass of functions
$\textnormal{\textbf{h}}\in\textnormal{\textbf{D}}$
for which $h_i(a)=h_i(b)=0$, $i=1,\ldots,N$).
Let us assume that the boundary conditions are \eqref{eq:BoundCond2}.
Condition \eqref{eq:tc:1stc} follows from \eqref{eq:1}.
Suppose now that the boundary conditions are \eqref{eq:BoundCond3}.
Then, there are two cases to consider.
(i) If $y_l(b)<y_l^b$, then there are admissible
neighboring paths with terminal value both above and below
$y_l(b)$, so that $h_l(b)$ can take either sign.
Therefore, the transversality condition is \eqref{eq:tc:1stc}.
(ii) Let $y_l(b)=y_l^b$. In this case neighboring paths
with terminal value $\tilde{y}_l\leq y_l(b)$ are considered.
Choose $h_l$ such that $h_l(b)\geq 0$. Then, $\epsilon \leq 0$
and the transversality condition, which has its root in the first order condition
\eqref{eq:2}, must be changed to an inequality. We obtain \eqref{eq:tc:2ndc}.
\end{proof}

When the Lagrangian does not depend on fractional derivatives,
then the left hand side of \eqref{eq:tc:1stc} and \eqref{eq:tc:2ndc}
reduce to the classical expression
$\partial_{N+l+1} L\left(x,\mathbf{y}(x),\mathbf{y}'(x)\right)$
(for instance, when $N = 1$ and $y(a)$ is fixed with $y(b)$ free,
then we get the well known boundary condition
$\partial_{3} L\left(b,y(b),y'(b)\right) = 0$).
In the particular case when the Lagrangian
does not depend on the classical derivatives,
$\gamma = 1$, $N = 1$,
and we have boundary conditions \eqref{eq:BoundCond2},
then one obtains from Theorem~\ref{theorem:nbc}
the following result of \cite{Agrawal}.

\begin{corollary}[\textrm{cf.} Theorem~1 of \cite{Agrawal}]
If $y$ is a local minimizer to problem
\begin{gather*}
\mathcal{J}(y)=\int_a^b L\left(x,y(x),{^C_a D_x^{\alpha}} y(x)\right) dx
\longrightarrow \min\\
y(a) = y_a \quad (y(b) \text{ is free}),
\end{gather*}
then $y$ satisfies the fractional Euler--Lagrange equation \eqref{eq:fEL:cC}.
Moreover,
\begin{equation*}
\left[{_x I_b^{1-\alpha}}\partial_3
L(x,y(x),{^C_a D_x^{\alpha}} y(x))\right]_{x=b}=0.
\end{equation*}
\end{corollary}


\subsection{The isoperimetric problem}
\label{sub:sec:OC}

We now consider the following problem of the calculus of variations.

\begin{problem}
\label{problem:P2}
Minimize functional \eqref{eq:Funct} subject
to given boundary conditions \eqref{eq:BoundCond}
and $r$ isoperimetric constraints
\begin{equation}
\label{eq:BoundCond4} \mathcal{G}^j(y) =\int\limits_a^b
G^j\left[y\right]^{\alpha,\beta}_{\gamma}(x)dx =\xi_j, \quad
j=1,\ldots,r,
\end{equation}
where $G^j\in C^1\left([a,b]\times\mathbb{R}^{3N};\mathbb{R}\right)$
and $\xi_j\in\mathbb{R}$ for $j=1,\ldots,r$.
\end{problem}

Problems of the type of Problem~\ref{problem:P2},
where some integrals are to be given a fixed value
while another one is to be made a maximum or a minimum,
are called \emph{isoperimetric problems}.
Such variational problems have found a broad class of important applications
throughout the centuries, with numerous useful implications
in astronomy, geometry, algebra, analysis, and engineering.
For references and recent advancements on the subject,
we refer the reader to \cite{isoJMAA,isoNabla,iso:ts,specialAveiro2Basia}.
Here, in order to obtain necessary optimality conditions
for the combined fractional isoperimetric problem
(Problem~\ref{problem:P2}), we make use of the following theorem.

\begin{theorem}[see, \textrm{e.g.}, Theorem~2 of \cite{G:H} on p.~91]
\label{theorem:nec} Let
$\mathcal{J},\mathcal{G}^{1},\ldots,\mathcal{G}^r$ be functionals
defined in a neighborhood of $y$ and having continuous first
variations in this neighborhood. Suppose that $y$ is a local
minimizer to the isoperimetric problem given by \eqref{eq:Funct},
\eqref{eq:BoundCond} and \eqref{eq:BoundCond4}. Assume that there
are functions $\textnormal{\textbf{h}}^1,
\ldots,\textnormal{\textbf{h}}^r\in\textnormal{\textbf{D}}$ such
that
\begin{equation}
\label{eq:reg:cond}
A=(a_{kl}), \quad a_{kl}:=\delta\mathcal{G}^k(\textbf{y};\textbf{h}^l),
\text{ has maximal rank } r.
\end{equation}
Then there exist constants $\lambda_1,\ldots,\lambda_r \in\mathbb{R}$
such that the functional
\begin{equation*}
\mathcal{F}
:=\mathcal{J}-\sum\limits_{j=1}^{r}\lambda _{j}\mathcal{G}^{j}
\end{equation*}
satisfies
\begin{equation}
\label{eq:var}
\delta\mathcal{F}(\mathbf{y};\mathbf{h})=0
\end{equation}
for all $\mathbf{h}\in \mathbf{D}$.
\end{theorem}

\begin{theorem}
\label{theorem:isop} Let assumptions of Theorem~\ref{theorem:nec}
hold. If $y$ is a local minimizer to Problem~\ref{problem:P2}, then
$y$ satisfies the system of $N$ fractional Euler--Lagrange equations
\begin{equation}
\label{eq:EL2}
\partial_i F _{\lambda}\left\{y\right\}_{\gamma}^{\alpha,\beta}(x)
-\frac{d}{dx}\partial_{N+i} F
_{\lambda}\left\{y\right\}_{\gamma}^{\alpha,\beta}(x)
+D^{\beta,\alpha}_{1-\gamma}\partial_{2N+i} F
_{\lambda}\left\{y\right\}_{\gamma}^{\alpha,\beta}(x)=0,
\end{equation}
$i=2,\ldots,N+1$, for all $x\in[a,b]$, where function
$F:[a,b]\times\mathbb{R}^{3N}\times\mathbb{R}^r\rightarrow\mathbb{R}$
is defined by
\begin{equation*}
F_{\lambda}\left\{y\right\}_{\gamma}^{\alpha,\beta}(x)
:=L\left[y\right]^{\alpha,\beta}_{\gamma}(x)
-\sum\limits_{j=1}^r\lambda_jG^j\left[y\right]^{\alpha,\beta}_{\gamma}(x).
\end{equation*}
\end{theorem}

\begin{proof}
Under assumptions of Theorem~\ref{theorem:nec},
the equation \eqref{eq:var} is fulfilled for every
$\textnormal{\textbf{h}}\in\textbf{D}$.
Consider a function $\textbf{h}$ such that
$\textnormal{\textbf{h}}(a)=\textnormal{\textbf{h}}(b)=0$. Then,
\begin{equation*}
\begin{split}
0 &=\delta\mathcal{F}\left(y;\textnormal{\textbf{h}}\right)
=\left.\frac{\partial}{\partial\epsilon}\mathcal{F}\left(y
+\epsilon\textnormal{\textbf{h}}\right)\right|_{\epsilon=0}\\
&=\int_a^b\left[\sum\limits_{i=2}^{N+1}\partial_i\
F_{\lambda}\left\{y\right\}_{\gamma}^{\alpha,\beta}(x)h_{i-1}(x)\right.
+\sum\limits_{i=2}^{N+1}\partial_{N+i}
F_{\lambda}\left\{y\right\}_{\gamma}^{\alpha,\beta}(x)
\frac{d}{dx}h_{i-1}(x)\\
&\qquad\qquad\left.+\sum\limits_{i=2}^{N+1}\partial_{2N+i}
F_{\lambda}\left\{y\right\}_{\gamma}^{\alpha,\beta}(x)
^{C}D^{\alpha,\beta}_{\gamma}h_{i-1}(x)\right]dx.
\end{split}
\end{equation*}
Using the classical and the integration by parts formula \eqref{eq:fracIP:C},
and applying the fundamental lemma of the calculus of variations
in a similar way as in the proof of Theorem~\ref{theorem:EL},
we obtain \eqref{eq:EL2}.
\end{proof}

Suppose now, that constraints \eqref{eq:BoundCond4}
are characterized by inequalities
\begin{equation*}
\mathcal{G}(y) =\int\limits_a^b
G^j\left[y\right]^{\alpha,\beta}_{\gamma}(x)dx \leq \xi_j, \quad
j=1,\ldots,r.
\end{equation*}
In this case we can set
\begin{equation*}
\int\limits_a^b \left(G^j\left[y\right]^{\alpha,\beta}_{\gamma}(x)
-\frac{\xi_j}{b-a}\right)dx+\int\limits_a^b\left(\phi_j(x)\right)^2dx=0,
\end{equation*}
$j=1,\ldots,r$, where $\phi_j$ have the same continuity properties
as $y_i$. Therefore, we obtain the following problem: minimize the
functional
\begin{equation*}
\hat{\mathcal{J}}(y)
=\int\limits_a^b\hat{L}\left(x,\textbf{y}(x),y'(x),
^{C}D^{\alpha,\beta}_{\gamma}\textbf{y}(x),\phi(x)\right)dx,
\end{equation*}
where $\phi(x)=\left[\phi_1(x),\ldots,\phi_r(x)\right]$,
subject to $r$ isoperimetric constraints
\begin{equation*}
\int\limits_a^b \left[G^j\left[y\right]^{\alpha,\beta}_{\gamma}(x)
-\frac{\xi_j}{b-a}+\left(\phi_j(x)\right)^2\right]dx=0, \quad
j=1,\ldots,r,
\end{equation*}
and boundary conditions \eqref{eq:BoundCond}. Assuming that assumptions
of Theorem~\ref{theorem:isop} are satisfied,
we conclude that there exist constants
$\lambda_j\in\mathbb{R}$, $j=1,\ldots,r$,
for which the system of $N$ equations
\begin{equation}
\label{eq:EL3}
\begin{split}
&\partial_i\hat{F}\left(x,y(x),y'(x),
^{C}D^{\alpha,\beta}_{\gamma}\textbf{y}(x),\lambda_1,\ldots,\lambda_r,\phi(x)\right)\\
&\quad-\frac{d}{dx}\partial_{N+i}\hat{F}\left(x,y(x),y'(x),
^{C}D^{\alpha,\beta}_{\gamma}y(x),\lambda_1,\ldots,\lambda_r,\phi(x)\right)\\
&\quad+D^{\beta,\alpha}_{1-\gamma}\partial_{2N+i}
\hat{F}\left(x,y(x),y'(x),
^{C}D^{\alpha,\beta}_{\gamma}\textbf{y}(x),\lambda_1,\ldots,\lambda_r,\phi(x)\right)
=0,
\end{split}
\end{equation}
$i=2,\ldots,N+1$, where $\hat{F}=\hat{L}+\sum\limits_{j=1}^r
\lambda_j\left(G^j-\frac{\xi_j}{b-a}+\phi_j^2\right)$ and
\begin{equation}
\label{eq:6}
\lambda_j\phi_j(x)=0, \quad j=1,\ldots,r,
\end{equation}
hold for all $x\in[a,b]$. Note that it is enough to assume that the
regularity condition \eqref{eq:reg:cond} holds for the constraints
which are active at the local minimizer $\textbf{y}$. Indeed,
suppose that $l<r$ constraints, say
$\mathcal{G}^1,\ldots,\mathcal{G}^l$ for simplicity, are active at
the local minimizer $\textbf{y}$, and there are functions
$\textnormal{\textbf{h}}^1,
\ldots,\textnormal{\textbf{h}}^l\in\textnormal{\textbf{D}}$ such
that the matrix $B=(b_{kj})$, $b_{k,j}
:=\delta\mathcal{G}^k\left(y;\textnormal{\textbf{h}}^j\right)$,
$k,j=1,\dots,l<r$ has maximal rank $l$. Since the inequality
constraints $\mathcal{G}^{l+1},\ldots,\mathcal{G}^r$ are inactive,
the condition \eqref{eq:6} is trivially satisfied by taking
$\lambda_{l+1}=\ldots=\lambda_r=0$. On the other hand, since the
inequality constraints $\mathcal{G}^1,\ldots,\mathcal{G}^l$ are
active and satisfy the regularity condition \eqref{eq:reg:cond} at
$\textbf{y}$, the conclusion that there exist constants
$\lambda_j\in\mathbb{R}$, $j=1,\ldots,r$, such that \eqref{eq:EL3}
holds, follow from Theorem~\ref{theorem:isop}. Moreover,
\eqref{eq:6} is trivially satisfied for $j=1,\ldots,l$.


\section{An Illustrative Example}
\label{sec:ex}

Let $\alpha\in\left(0,1\right)$, $N = 1$, $\gamma=1$,
and $\xi\in\mathbb{R}$. Consider the following
fractional isoperimetric problem:
\begin{equation}
\label{eq:ex}
\begin{gathered}
\mathcal{J}(y)=\int_0^1\left(y'(x)
+ \, {^C_{0}\textsl{D}_x^{\alpha}} y(x)\right)^2dx \longrightarrow \min\\
\mathcal{G}(y)=\int_0^1\left(y'(x)
+ \, {^C_{0}\textsl{D}_x^{\alpha}} y(x)\right)dx = \xi\\
\begin{split}
y(0)=0\, , \
y(1)=&\int_0^1 E_{1-\alpha}\left(-\left(1-t\right)^{1-\alpha}\right)\xi dt.
\end{split}
\end{gathered}
\end{equation}
In this problem we make use of the Mittag-Leffler function
$E_{\alpha}(z)$. We recall that the Mittag-Leffler function is
defined by
\begin{equation*}
E_{\alpha}(z) =\sum_{k=0}^\infty\frac{z^k}{\Gamma(\alpha k+1)}\, .
\end{equation*}
This function appears naturally in the solution
of fractional differential equations,
as a generalization of the exponential function
\cite{CapelasOliveira}.
Indeed, while a linear second order
ordinary differential equation
with constant coefficients presents an exponential function as solution,
in the fractional case the Mittag--Leffler functions emerge \cite{Kilbas}.

In our example \eqref{eq:ex} the function $F$
of Theorem~\ref{theorem:isop} is given by
$$
F(x,y,y', {^C_{0}\textsl{D}_x^{\alpha}} y,\lambda)
=\left(y'+{^C_{0}\textsl{D}_x^{\alpha}} y\right)^2 -\lambda \left(y'
+ {^C_{0}\textsl{D}_x^{\alpha}} y\right).
$$
One can easily check that $y$ such that
\begin{equation}
\label{eq:y:ex} y(x)=\int_0^x
E_{1-\alpha}\left(-\left(x-t\right)^{1-\alpha}\right) \xi dt
\end{equation}
\begin{itemize}
\item is not an extremal for $\mathcal{G}$;
\item satisfies $y'+ {^C_{0}\textsl{D}_x^{\alpha}}y= \xi$
(see, \textrm{e.g.}, \cite[p.~328, Example~5.24]{Kilbas}).
\end{itemize}
Moreover, \eqref{eq:y:ex} satisfies the Euler--Lagrange equations
\eqref{eq:EL2} for $\lambda=2\xi$, \textrm{i.e.},
$$
-\frac{d}{dx}\left(2\left(y' + {^C_{0}\textsl{D}_x^{\alpha}}
y\right) -2\xi\right) + {_{x}\textsl{D}_1^{\alpha}}\left(2\left(y'+
{^C_{0}\textsl{D}_x^{\alpha}} y\right) -2\xi\right)=0.
$$
We conclude that \eqref{eq:y:ex} is an extremal for the
isoperimetric problem \eqref{eq:ex}.

\begin{remark}
When $\alpha \rightarrow 1$ the isoperimetric constraint is
redundant with the boundary conditions, and the fractional
isoperimetric problem \eqref{eq:ex} simplifies to the classical
variational problem
\begin{equation}
\label{eq:ex:alpha1}
\begin{gathered}
\mathcal{J}(y)
=4\int_0^1 (y'(x))^2 dx \longrightarrow \min\\
y(0)=0  \, , \quad y(1) = \frac{\xi}{2}.
\end{gathered}
\end{equation}
Our fractional extremal \eqref{eq:y:ex} gives $y(x)=\frac{\xi}{2}x$
for $i=1,\ldots,N$, which is exactly the minimizer of
\eqref{eq:ex:alpha1}.
\end{remark}

\begin{remark}
Choose $\xi = 1$. When $\alpha \rightarrow 0$
one gets from \eqref{eq:ex} the classical isoperimetric problem
\begin{equation}
\label{eq:ex:alpha0}
\begin{gathered}
\mathcal{J}(y) =\int_0^1\left(y'(x)
+ y(x)\right)^2 dx \longrightarrow \min\\
\mathcal{G}(y)
=\int_0^1 y(x) dx = \frac{1}{e}\\
y(0)=0 \quad y_i(1)= 1-\frac{1}{\mathrm{e}}.
\end{gathered}
\end{equation}
Our extremal \eqref{eq:y:ex} is then reduced to the classical
extremal $y(x)=1 - \mathrm{e}^{-x}$ of the isoperimetric problem
\eqref{eq:ex:alpha0}.
\end{remark}

\begin{remark}
Let $\alpha=\frac{1}{2}$.
Then \eqref{eq:ex} gives the following
fractional isoperimetric problem:
\begin{equation}
\label{eq:ex:alpha=1/2}
\begin{gathered}
\mathcal{J}(y)=\int_0^1\left(y'(x)
+  {^C_{0}\textsl{D}_x^{\frac{1}{2}}} y(x)\right)^2 dx\longrightarrow \min\\
\mathcal{G}(y)=\int_0^1\left(y'(x)
+ {^C_{0}\textsl{D}_x^{\frac{1}{2}}} y(x)\right)dx=\xi \\
y(0) =0\, , \quad y(1) = \xi\left(
\mathrm{erfc}(1)\mathrm{e}+\frac{2}{\sqrt{\pi}}-1\right),
\end{gathered}
\end{equation}
where $\mathrm{erfc}$ is the complementary error function defined by
\begin{equation*}
\mathrm{erfc}(z)=\frac{2}{\sqrt{\pi}}\int_{z}^{\infty}exp(-t^2)dt.
\end{equation*}
The extremal \eqref{eq:y:ex} for the particular fractional
isoperimetric problem \eqref{eq:ex:alpha=1/2} is
$$
y(x)=\xi\left(\mathrm{e}^x \mathrm{erfc}(\sqrt{x})
+\frac{2\sqrt{x}}{\sqrt{\pi}}-1\right).
$$
\end{remark}


\section*{Acknowledgements}

This work was partially supported by the
\emph{Portuguese Foundation for Science and Technology} (FCT)
through the \emph{Center for Research and Development in Mathematics and Applications} (CIDMA).
TO is also supported by FCT through the Ph.D. fellowship SFRH/BD/33865/2009;
ABM by BUT grant S/WI/1/08; and DFMT through the project UTAustin/MAT/0057/2008.



\end{document}